\documentclass[12pt,leqno,a4paper]{amsart}
\usepackage{amssymb,enumerate,verbatim}                  

\usepackage[usenames]{color}
\usepackage{hyperref}   
\usepackage{array,arydshln}
\usepackage{enumitem}
\usepackage{comment}
\usepackage[all]{xy}

\setlist[enumerate]{labelsep=*, leftmargin=1.5pc}
\setlist[itemize]{labelsep=0.5pc}

\hypersetup{pdftex,                           
bookmarks=true,
pdffitwindow=true,
colorlinks=true,
citecolor=black,
filecolor=black,
linkcolor=black,
urlcolor=blue,
hypertexnames=true}

\DeclareMathOperator{\Hom}{Hom}
\DeclareMathOperator{\Ext}{Ext}%

\DeclareMathOperator{\Syl}{Syl}    

\newcommand{\PGL}{{\operatorname{PGL}}}
\newcommand{\PSL}{{\operatorname{PSL}}}

\newcommand{\SL}{{\operatorname{SL}}}

\newcommand{\SU}{{\operatorname{SU}}}

\newcommand{\IZ}{\mathbb{Z}}

\newcommand{\fA}{{\mathfrak{A}}}

\newcommand{\cO}{{\mathcal{O}}}

\let\lra=\longrightarrow

\newtheorem{thm}{Theorem}[section]

\newtheorem{lem}[thm]{Lemma}
\newtheorem{prop}[thm]{Proposition}

\theoremstyle{definition}

\newtheorem{rem}[thm]{Remark}

\theoremstyle{remark}

\begin{document}


\title{Splendid Morita equivalences for principal blocks with generalised quaternion defect groups}
\date{\today}
\author{{Shigeo Koshitani and Caroline Lassueur}}
\address{{\sc Shigeo Koshitani}, Center of Frontier Science,
Chiba University, 1-33 Yayoi-cho, Inage-ku, Chiba 263-8522, Japan.}
\email{koshitan@math.s.chiba-u.ac.jp}
\address{{\sc Caroline Lassueur}, FB Mathematik, TU Kaiserslautern, Postfach 3049,
         67653 Kaiserslautern, Germany.}
\email{lassueur@mathematik.uni-kl.de}

\thanks{The first author was partially supported by the Japan Society for
Promotion of Science (JSPS), Grant-in-Aid for Scientific Research
(C)15K04776, 2015--2017. The second author acknowledges financial support by DFG TRR 195. This paper is also based upon work supported by the National Science Foundation under Grant No. DMS1440140 while the second author was in residence at the Mathematical Sciences Research Institute in Berkeley, California, during the Spring 2018 semester.}




\begin{abstract}
We prove that splendid Morita equivalences between principal blocks  of finite groups with dihedral Sylow $2$-subgroups realised by Scott modules can be lifted to splendid Morita equivalences between principal blocks of finite groups with generalised quaternion Sylow $2$-subgroups realised by Scott modules.
\end{abstract}

\maketitle

\pagestyle{myheadings}
\markboth{S. Koshitani and C. Lassueur}{Splendid Morita equivalences for principal $2$-blocks with generalised quaternion defect groups}

\section{Introduction}

The aim of this note is to prove that the knowledge of the equivalence classes of splendid Morita equivalences for principal blocks with dihedral defect groups is enough to describe  
the splendid Morita equivalence classes of  principal blocks with generalised quaternion defect groups, as well as the bimodules realising these equivalences. Our main result is as follows:

\begin{thm}\label{MainThm}\label{thm:intro}
Let $G$ be an arbitrary finite group with a generalized quaternion
Sylow $2$-subgroup $P\cong Q_{2^n}$ of order $2^n$ with $n\geq 3$. Let $k$ be an algebraically closed field of characteristic~$2$, and let $B_0(kG)$ denote the principal block of $kG$.  Then: 
\begin{enumerate}
\item[\rm(a)] $B_0(kG)$ is splendidly Morita equivalent to precisely one of the following principal blocks:
\begin{itemize}
\item[\rm(1)] $kQ_{2^n}$;
\item[\rm(2)] $B_0(k[2.\mathfrak A_7])$ in case $n=4$;
\item[\rm(3)] $B_0(k[{\mathrm{SL}}_2(q)])$, 
where $q$ is a fixed odd prime power such that $(q-1)_2=2^n$;
\item[\rm(4)] $B_0(k[{\mathrm{SL}}_2(q)])$, 
where $q$ is a fixed odd prime power such that $(q+1)_2=2^n$;
\item[\rm(5)] $B_0(k[{\mathrm{2.PGL}}_2(q)])$, 
where $q$ is a fixed odd prime power such that $2(q-1)_2=2^n$;
\item[\rm(6)] $B_0(k[{\mathrm{2.PGL}}_2(q)])$, 
where $q$ is a  fixed odd prime power such that \smallskip $2(q+1)_2=2^n$.
\end{itemize}
\item[\rm(b)] Moreover, the splendid Morita equivalence of Part {\rm(a)} is realised by the Scott module ${\mathrm{Sc}}(G\times H, \Delta P)$, where $H$ is the group listed in the corresponding case  \smallskip $(1)$--$(6)$.
\item[\rm(c)]
In particular, if $q$ and $q'$ are two odd prime powers as in cases {\rm(3)}--{\rm(6)} such that
$(q-1)_2=(q'-1)_2$, respectively $(q+1)_2=(q'+1)_2$, then  $B_0(k[{\mathrm{SL}}_2(q)])$ is splendidly Morita equivalent to
$B_0(k[{\mathrm{SL}}_2(q')])$, respectively $B_0(k[{\mathrm{2.PGL}}_2(q)])$ is splendidly Morita equivalent to $B_0(k[{\mathrm{2.PGL}}_2(q')])$. 
\end{enumerate}
\end{thm}

This is motivated by  Puig's Finiteness Conjecture (see Th\'{e}venaz \cite[(38.6) Conjecture]{The95} for a published version) stating that for a given prime $p$ and a finite $p$-group $P$ there are only finitely many isomorphism classes of interior $P$-algebras arising as source algebras of $p$-blocks of finite groups with defect groups isomorphic to $P$, or equivalently that there are only a finite number of {\it splendid Morita equivalence}  classes of blocks of finite groups with defect groups isomorphic to~$P$. This obviously strengthens Donovan's Conjecture.  However, we emphasise that by contrast to Donovan's Conjecture,  if $p$ is a prime number, $(K,\cO, k)$ a $p$-modular system with $k$ algebraically closed, and  Puig's Finiteness Conjecture holds over $k$, then it automatically holds over~$\cO$, since the bimodules inducing splendid Morita equivalences are liftable from $k$ to $\cO$. The cases where $P$ is either cyclic \cite{Lin96b} or a Klein-four group \cite{CEKL11} are the only case where this conjecture has been proved to hold in full generality. Else, under additional assumptions, Puig's Finiteness Conjecture has also been proved for several classes of finite groups, as for instance for $p$-soluble groups \cite{Pui94}, for symmetric groups \cite{Pui94}, for alternating groups and  the double covers thereof,  for Weyl groups, or for classical groups, see \cite{HK00, HK05} and the references therein.\par
We also recall that Erdmann \cite{Erd90} classified blocks of tame representation type up to Morita equivalence by describing their  basic algebras by generators and relations making intense use of the Auslander-Reiten quiver, but her results are not liftable to $\cO$ in general and do not imply that the resulting Morita equivalences are necessarily splendid Morita equivalences.


\vspace{2mm}
\section{Preliminaries}

Throughout this paper, unless otherwise stated we adopt the following notation and conventions.  All groups considered are assumed to be finite and all modules over finite group algebras are assumed to be finitely generated right modules.\par
We denote the dihedral group of order $2^n$ ($n\geq 2$) by $D_{2^n}$, the generalised quaternion group of order $2^n$ ($n\geq 3$) by  $Q_{2^n}$, and the cyclic group of order $m\in\IZ_{\geq 0}$ in multiplicative notation by $C_m$. 
Given an arbitrary finite group $G$  of order divisible by a prime $p$, we denote by $k_G$ the trivial $kG$-module, and we write $B_0(kG)$ for the principal $p$-block of~$kG$. \par
For a subgroup $H\leq G$ we denote the Scott $kG$-module with respect to~$H$ by ${\mathrm{Sc}}(G,H)$.
By definition ${\mathrm{Sc}}(G,H)$ is, up to isomorphism, the unique indecomposable direct summand
of the induced module ${k_H}{\uparrow}^G$ which contains $k_G$ in its
top (or equivalently in its socle).  If $Q\in \Syl_p(H)$, then $Q$ is a vertex of ${\mathrm{Sc}}(G,H)$ 
and a $p$-subgroup of $G$ is a vertex of ${\mathrm{Sc}}(G,H)$ if and only if it is $G$-conjugate to $Q$. 
It follows that ${\mathrm{Sc}}(G,H)={\mathrm{Sc}}(G,Q)$. We refer the reader to  \cite[\S2]{Bro85} and  \cite[Chap.4 \S 8.4]{NT89} for these results. Equivalently, ${\mathrm{Sc}}(G,H)$ is the relative $H$-projective cover of the trivial module $k_G$;  see \cite[Proposition 3.1]{The85}. 
If $N\vartriangleleft G$ is a normal subgroup, then we use the bar notation 
$$\bar{G}:=G/N\qquad\text{ and }\qquad \bar g:=gN\text{ for }g\in G$$
to denote the corresponding quotient and its elements. Moreover, if $M$ is a $kG$-module such that $N\leq \ker M$, then $M$ becomes a $k\bar G$-module via  $m\,{\bar g}:=mg$ for each $g\in G$ and each $m\in M$, and $M_{kG}$ is indecomposable if and only if $M_{k\bar G}$ is indecomposable. Moreover, 
$$M\otimes_{kG}\,k\bar G \cong M$$
as right $k\bar G$-modules, canonically via the map $m\otimes_{kG}\bar g\mapsto m\bar g$ for all $m\in M$, $g\in G$. 
We denote by $\pi: G\twoheadrightarrow \bar{G}$ the canonical homomorphism, and extended it by $k$-linearity to $\pi: kG\twoheadrightarrow k\bar G$. Then for  a block  $A$ of $kG$, we write $\bar A:=\pi (A)$, which, in general, is not  a block but a certain direct sums of blocks of $k\bar G$.\\

If $G$ and $H$ are finite groups,  $A$ and $B$ are block algebras of $kG$ and $kH$ respectively and $M$ is an indecomposable $(A,B)$-bimodule  inducing a Morita equivalence between $A$ and $B$, then we view $M$ as a right $k[G\times H]$-module via the right $G\times H$-action $m\cdot (g,h):=g^{-1}mh$ for all $m\in M, g\in G, h\in H$.
Furthermore, the algebras $A$ and $B$ are called \emph{splendidly Morita equivalent} (or also \emph{Puig equivalent}), if  
there is a Morita equivalence between $A$ and $B$ induced by an $(A,B)$-bimodule $M$ such that $M$, viewed  as a right $k[G\times H]$-module, is a $p$-permutation module. 
In this case, due to a result of Puig (see \cite[Corollary~7.4]{Pui99} and \cite[Proposition~9.7.1]{Lin18}), the defect groups $P$ and $Q$ of $A$ and $B$ respectively are isomorphic (and hence from now on we identify $P$ and $Q$).
Obviously $M$ is indecomposable as a $k(G\times H)$-module. Moreover, since $M$ induces a Morita equivalence, $_AM$ and $M_B$ are both projective and therefore $M$ has a vertex $R$ of the form  $R=\Delta (P) \leq G\times H$.  We notice that, by a result of Puig and Scott, this definition  is equivalent to the condition that $A$ and $B$ have source algebras which are isomorphic as interior $P$-algebras (see \cite[Theorem~4.1]{Lin01} and \cite[Remark 7.5]{Pui99}).\par In order to produce splendid Morita equivalences between principal 
blocks of two finite groups $G$ and $G'$ with a common defect group $P$, we mainly use Scott modules of the form 
${\mathrm{Sc}}(G\times G', \, \Delta P)$, which are obviously $(B_0(kG),B_0(kG'))$-bimodules.  In particular, we note that if  
the Scott module $M:=\textrm{Sc}(G\times H,\Delta P)$ induces a Morita equivalence between 
the principal blocks $A$ and $B$ of $kG$ and $kH$, respectively, then this is a splendid Morita equivalence because Scott modules are $p$-permutation modules by definition.\\

In this paper, we will rely on the following classification of  principal 2-blocks of groups with dihedral Sylow $2$-subgroups, up to splendid Morita equivalence, where the result for Klein-four groups is the main result of \cite{CEKL11} and the result for dihedral groups of order at least $8$ is the main result of \cite{KL18}:

\begin{thm}[\cite{CEKL11} and \cite{KL18}]\label{dihedral}
Let $G$ be an arbitrary finite group with a dihedral Sylow $2$-subgroup
$P=D_{2^n}$ of order $2^n$ with $n\geq 2$.  Let $k$ be an algebraically closed field of characteristic~$2$, and let $B_0(kG)$ denote the principal block of $kG$.  Then: 
\begin{enumerate}[leftmargin=9mm]
\item[\rm(a)] $B_0(kG)$ is splendidly Morita equivalent to precisely one of the following principal blocks:
\begin{itemize}
\item[\rm(1)] $kD_{2^n}$;
\item[\rm(2)] $B_0(k\mathfrak A_7)$ in case $n=3$;
\item[\rm(3)] $B_0(k[{\mathrm{PSL}}_2(q)])$, 
where $q$ is a fixed odd prime power such that $(q-1)_2=2^n$;
\item[\rm(4)] $B_0(k[{\mathrm{PSL}}_2(q)])$, 
where $q$ is a fixed odd prime power such that $(q+1)_2=2^n$;
\item[\rm(5)] $B_0(k[{\mathrm{PGL}}_2(q)])$, 
where $q$ is a fixed odd prime power such that $2(q-1)_2=2^n$; 
\item[\rm(6)] $B_0(k[{\mathrm{PGL}}_2(q)])$, 
where $q$ is a  fixed odd prime power such that \smallskip $2(q+1)_2=2^n$.
\end{itemize}
\item[\rm(b)] The splendid Morita equivalence of {\rm(a)} is realised by the Scott module  ${\mathrm{Sc}}(G\times H, \Delta P)$, where $H$ is the group listed in the corresponding case  \smallskip $(1)$--$(6)$.
\item[\rm(c)] If $q$ and $q'$ are two odd prime powers
as in {\rm(3)}-{\rm(6)} such that either 
$(q-1)_2=(q'-1)_2$ (Cases {\rm(3)} and {\rm(5)}), respectively $(q+1)_2=(q'+1)_2$ (Cases {\rm(4)} and {\rm(6)}), then 
$B_0(k[{\mathrm{PSL}}_2(q)])$ is splendidly Morita equivalent to
$B_0(k[{\mathrm{PSL}}_2(q')])$, respectively $B_0(k[{\mathrm{PGL}}_2(q)])$ is splendidly Morita equivalent to $B_0(k[{\mathrm{PGL}}_2(q')])$. 
\end{enumerate}
\end{thm} 

In \cite{KL18} the  proof of the statements  for $n\geq 3$ partly  relies on the Brauer indecomposability of the Scott module  ${\mathrm{Sc}}(G\times H, \Delta P)$ inducing the splendid Morita equivalence. 
Passing to generalised quaternion Sylow $2$-subgroup we won't need to use arguments involving Brauer indecomposability, however we note that the other way around a Scott module realising a splendid Morita equivalence is necessarily Brauer indecomposable. \\

Let us recall the definition of the {\it Brauer indecomposability}, which was first defined in \cite{KKM11}. 
Given a $kG$-module $V$ and a $p$-subgroup $Q\leq G$, the \emph{Brauer  construction} (or \emph{Brauer quotient}) of $V$ with  respect to $Q$  is defined to be the $kN_G(Q)$-module
$$
V(Q):= V^{Q}\big/ \sum_{R<Q}{\mathrm{Tr}}^Q_R(V^R)\, ,
$$
where, for $R\leq Q$, $V^{R}$ denotes the set of $R$-fixed points of $V$, and for each proper subgroup $R<Q$,  ${\mathrm{Tr}}^Q_R~:~V^R\longrightarrow V^Q, v\mapsto \sum_{xR\in Q/R}xv$ denotes the relative trace map. 
See e.g. \cite[\S 27]{The95}. The Brauer  construction  with  respect  to $Q$ sends  a $p$-permutation $kG$-module 
$V$ functorially to the $p$-permutation $kN_G(Q)$-module $V(Q)$, see \cite[p.402]{Bro85}.\\
Then, following the terminology introduced in \cite{KKM11}, a $kG$-module $V$ is said to be \emph{Brauer indecomposable}  if the $kC_G(Q)$-module $V(Q)\!\downarrow^{N_G(Q)}_{C_G(Q)}$ is indecomposable or zero for each $p$-subgroup $Q\leq G$.

\begin{lem}\label{BrauerIndec}
Let $G$ and $H$ be two arbitrary finite groups with a common Sylow $p$-subgroup~$P$. 
If $B_0(kG)$ and $B_0(kH)$ are splendidly Morita equivalent via the Scott module ${\mathrm{Sc}}(G\times H,\,\Delta P)$, 
then ${\mathrm{Sc}}(G\times H,\,\Delta P)$ is Brauer indecomposable. In particular the Scott module ${\mathrm{Sc}}(G\times H,\,\Delta P)$ in Theorem~\ref{thm:intro}(b) is Brauer indecomposable.
\end{lem}

\begin{proof}
Set $M:={\mathrm{Sc}}(G\times H,\,\Delta P)$, $B:=B_0(kG)$ and $b:=B_0(kH)$. By definition of the Brauer indecomposability  \cite[Section 1]{KKM11}, we have to prove that  the $kC_G(Q)$-module $M(Q)\!\downarrow^{N_G(Q)}_{C_G(Q)}$ is indecomposable or zero for each $p$-subgroup $Q\leq G$.  Now, by the assumption $M$ induces a Morita equivalence and hence
induces a stable equivalence of Morita type between $B_0(kG)$ and $B_0(kH)$. Thus, by a result of Brou\'e and Rouquier (see \cite[Lemma 4.1]{KL18}) $M(\Delta Q)$ induces a Morita equivalence between
$B_0(kC_G(Q))$ and $B_0(kC_H(Q))$ for any non-trivial subgroup $Q\leq P$. But since blocks are indecomposable as $k$-algebras, $M(\Delta Q)$ has to be indecomposable
as $(B_0(kC_G(Q)), B_0(kC_H(Q)))$-bimodule, hence  as $k(C_{G\times H}(\Delta Q))$-module.
Finally, it is obvious that $M(\Delta \langle 1\rangle)=M$, so that $M$ itself is indecomposable as $C_{G\times H}(\Delta\langle 1\rangle)$-module since  $C_{G\times H}(\Delta\langle 1\rangle)=G\times H$. The claim follows.
\end{proof}

\noindent Furthermore, we will use the following well-known properties of group cohomology. We sketch their proofs for completeness.
Below we write the cylic group of order $2$ as $C_2$, resp. $\IZ/2$, to emphasise that the group law is considered multiplicatively, resp. additively.

\begin{lem}\label{lem:cohomWithCoeffsinC2}
Let $G$ be a finite group, and let $P$ be a Sylow $2$-subgroup of $G$. Let $C_2$ be the cyclic group of order $2$, which we view  as both a trivial $G$-module and a trivial $P$-module. Then restriction in cohomology 
$$\mathrm{res}^G_P:H^2(G,C_2)\lra H^2(P,C_2)$$
is an injective group homomorphism.
\end{lem}

\begin{proof}
This follows directly  from the fact that  post-composition with the transfer is multiplication by the index $|G:P|$, which is an isomorphism as $|G:P|$ is prime to $2$ and  $2$ annihilates $H^2(G,C_2)$.
\end{proof}

\begin{lem}\label{lem:centralExtensionsD2n}
Let $n\geq 3$ be an integer. The following hold:
\begin{enumerate}
\item[\rm(a)] If $C_2$ denotes the cyclic group of order $2$, viewed as a trivial $kD_{2^{n-1}}$-module, then  $H^2(D_{2^{n-1}},C_2)\cong (\IZ/2)^3$.
\item[\rm(b)] There is a unique isomorphism class of central extensions
$$1\lra C_2\lra P\lra D_{2^{n-1}}\lra 1\,$$
such that $P\cong Q_{2^n}$.
\end{enumerate}
\end{lem}

\begin{proof}{\ }
\begin{enumerate}
\item[\rm(a)] The cohomological Universal  Coefficient Theorem yields
$$H^2(D_{2^{n-1}},C_2)\cong \Hom_{\IZ}(M_2(D_{2^{n-1}}),\IZ/2)\oplus \Ext^1_{\IZ}(D_{2^{n-1}}/[D_{2^{n-1}},D_{2^{n-1}}],\IZ/2)\,,$$
where $M_2(D_{2^{n-1}})\cong \IZ/2$ is the Schur multiplier of $D_{2^{n-1}}$ and $D_{2^{n-1}}/[D_{2^{n-1}},D_{2^{n-1}}]\cong \IZ/2\times\IZ/2$ is the abelianisation of $D_{2^{n-1}}$. Hence \smallskip $H^2(D_{2^{n-1}},C_2)\cong (\IZ/2)^3$.
\item[\rm(b)] The isomorphism classes of central extensions of the  form
$$1\lra C_2\lra P\lra D_{2^{n-1}}\lra 1\,.$$
are in bijection with $H^2(D_{2^{n-1}}, C_2)$,  hence by (a) there are $8$ isomorphism classes of such extensions. 
Since a presentation of $D_{2^{n-1}}$ is $\langle \rho,\sigma\mid \rho^2=1=\sigma^2, (\rho\sigma)^{2^{n-2}}=1\rangle$, obviously $P$ admits a presentation of the form
$$\langle r,s,t\mid rt=tr, st=ts, t^2=1, r^2=t^a, s^2=t^b, (rs)^{2^{n-2}}=t^c\rangle, \,\,a,b,c\in\{0,1\}\,.$$
Letting $a,b,c$ vary, we obtain the following groups $P$:
\begin{itemize}
\item[\rm(i)] The case $a=b=c=0$ gives the direct product $C_2\times D_{2^{n-1}}$.
\item[\rm(ii)] The case $a=b=0$, $c=1$ gives the dihedral group $D_{2^{n}}$.
\item[\rm(iii)] The cases $a=c=0$, $b=1$ and $b=c=0$, $a=1$ give the group $(C_{2^{n-2}}\times C_2):C_2$.
\item[\rm(iv)] The cases $a=0$, $b=c=1$ and $b=0$, $a=c=1$ both give the semi-dihedral group $SD_{2^{n}}$ of order $2^n$.
\item[\rm(v)]  The case $c=0$, $a=b=1$ gives the group $C_{2^{n-2}}:C_4$.
\item[\rm(vi)] The case $a=b=c=1$ gives the generalised quaternion group \smallskip $Q_{2^n}$.
\end{itemize}
If $n\geq 4$, the groups in cases {\rm(i)-(vi)} are pairwise non-isomorphic. If $n=3$ the above holds as well, but the groups in {\rm(ii)} and {\rm(iii)} are all isomorphic to $D_8$, and the groups in {\rm(iv)} and {\rm(v)} are all isomorphic to $C_2\times C_4$. The claim follows.
\end{enumerate}
\end{proof}


\section{Lifting splendid Morita equivalences from a central quotient}

Although our aim is to treat blocks with generalised quaternion defect groups. We first present some results that hold for an algebraically closed field of arbitrary characteristic $p>0$ and principal blocks with arbitrary defect groups.  Throughout this section, we assume that $G$ and $H$ are two arbitrary finite groups, we let $P$ denote a common $p$-subgroup of $G$ and $H$. Furthermore, given a  subgroup $Z$ of $P$ which is normal in both $G$ and $H$, we use the bar notation   with $\bar G:=G/Z$, $\bar H:=H/Z$ and $\bar P:=P/Z$.

\begin{lem}{\ }\label{lem:Scott}\label{lem:GxH}
\begin{enumerate}
\item[\rm(a)]Let $Q\leq G$ be an arbitrary subgroup, let $Z\leq Z(G)\cap Q$ and set $\bar Q:=Q/Z$ and $\bar G:=G/Z$. 
Then
$$
{\mathrm{Sc}}(G,Q)\otimes_{kG}\,k\bar G\ \cong \ {\mathrm{Sc}}(\bar G,\bar Q)
$$
as right \smallskip $k\bar G$-modules.
\item[\rm(b)] Let $Z\leq P$ such that  $Z\leq Z(G)\cap Z(H)$.
Then
$$
k\bar G\otimes_{kG}\,{\mathrm{Sc}}(G\times H,\Delta P)\otimes_{kH}\,k\bar H
\cong{\mathrm{Sc}}(\bar G\times\bar H,\Delta\bar P)
$$
as right $k(\bar G\times\bar H)$-modules.
\end{enumerate}
\end{lem}
\begin{proof}{\ }
\begin{enumerate}
\item[\rm(a)] Set $M:={\mathrm{Sc}}(G,Q)$. Clearly $Z\leq\ker(k_Q{\uparrow}^G)$, hence  $Z\leq\ker M$ as well, since by definition $M\mid {k_Q}{\uparrow}^G$. 
Hence 
$M$ may be considered as a $k\bar G$-module, and in fact $M\cong M\otimes_{kG}\,k\bar G$ is indecomposable as a $k\bar G$-module.
Since $M\mid k_Q{\uparrow}^G$,
it follows that 
\begin{equation*}
\begin{split}
M_{k\bar G} \cong  M\otimes_{kG}\,k\bar G \, \Big| \,  k_Q{\uparrow}^G\otimes_{kG}k\bar G &=k_Q\otimes_{kQ}\,kG\otimes_{kG}\,k\bar G\\
&\cong k_Q\otimes_{kQ}\,k\bar G \\
&\cong k_{\bar Q}\otimes_{k\bar Q}\,k\bar G = k_{\bar Q}{\uparrow}^{\bar G}.                        
\end{split}
\end{equation*}
Hence ${\mathrm{Sc}}(\bar G,\,\bar Q) \cong M\otimes_{kG}\,k\bar G$, since $M$ is indecomposable as a $kG$-module, and thus also as a $k\bar G$-module. 
\item[\rm(b)] Since
$$ 
k\bar G\otimes_{kG}\,{\mathrm{Sc}}(G\times H,\,\Delta P)\otimes_{kH}\,k\bar H
\cong
{\mathrm{Sc}}(G\times H,\,\Delta P)\otimes_{k(G\times H)}\,k(\bar G\times\bar H)
$$
as right $k(\bar G\times\bar H)$-modules, the assertion follows from (a).
\end{enumerate}
\end{proof}

\begin{lem}[{}{\cite[Lemma 10.2.11]{Rou98}}]{\ }\label{thm:rouquier}
Assume $Z$ is a subgroup of $P$ such that $Z\vartriangleleft G$ and $Z\vartriangleleft H$,  
and let $A$ and $B$ be blocks  of $kG$ and $kH$, respectively. 
Let $M$ be an indecomposable $(kG, kH)$-bimodule, which is a trivial source module  with vertex $\Delta P$ when regarded as a right $k(G\times H)$-module. Set $\bar M:=k\bar G\otimes_{kG}\,M\otimes_{kH}\,k\bar H$.
Then the following are equivalent:
\begin{itemize}
\item[\rm(i)] 
$M$ induces a Morita equivalence between the blocks $A$ and $B$;
\item[\rm(ii)]
$\bar M$ induces a Morita equivalence between $\bar A$ and $\bar B$.
\end{itemize}
\end{lem}
\noindent Note that in this lemma $\bar A$ and $\bar B$ are not necessarily blocks of $k\bar G$ and $k\bar H$, but sums of blocks.

\begin{proof}
 Apply \cite[Lemma 10.2.11]{Rou98} with 
$Z=:R$, $A=:\mathcal OGe$, $B=:\mathcal OHf$,
$\bar A=:\mathcal O\bar G\bar e$, $\bar B=: \mathcal O\bar H\bar f$,
$M=: C$, $\bar M=:\bar C$,
$\Delta P=:Q$, and  replace ``a Rickard complex'' with 
``a bimodule inducing the Morita equivalence''. Then, we see that
$$\Delta P \cap (\{1\}\times H)=\Delta P\cap(G\times\{1\})=\{1\}$$
and
$$Z\times Z\leq (Z\times\{1\})\Delta P= (\{1\}\times Z)\Delta P\,,$$ since $(z,z')=(z{z'}^{-1},1)(z',z')$, $(z,z')=(1,z'z^{-1})(z,z)$ and
$(z,1)(u,u)=(1,z^{-1})(zu,zu)$ for all $z,z'\in Z$ and all $u\in P$. The claim follows.
\end{proof}

\begin{prop}\label{thm:liftingNew}
Assume $Z\leq P$ such that $Z\leq Z(G)\cap Z(H)$. Let  
$A$ and $B$ be blocks of $kG$ and $kH$, respectively.  
Let $M$ be an indecomposable $(A, B)$-bimodule, which is a trivial source module with vertex $\Delta P$   when regarded as a right $k(G\times H)$-module. 
Set $\bar M:=k\bar G\otimes_{kG}\,M\otimes_{kH}\,k\bar H$. Then:
\begin{enumerate}
\item[\rm (a)] The following statements are equivalent:
   \begin{itemize}
\item[\rm(i)] $M$ induces a Morita equivalence between the blocks $A$ and $B$.
\item[\rm(ii)] $\bar M$ induces a Morita equivalence between the blocks $\bar A$ and \smallskip $\bar B$.
   \end{itemize}
\item[\rm (b)] In particular,    the following statements are equivalent:
   \begin{itemize}
\item[\rm(i)] ${\mathrm{Sc}}(G\times H,\Delta P)$ induces a Morita equivalence between $B_0(kG)$ and $B_0(kH)$.
\item[\rm(ii)] ${\mathrm{Sc}}(\bar G\times\bar H,\Delta \bar P)$  induces a Morita equivalence between $B_0(k\bar G)$ and $B_0(k\bar H)$.
   \end{itemize}
\end{enumerate}
\end{prop}

\begin{proof}{\ }
\begin{enumerate}
\item[\rm (a)] First, we note that $\bar A$ and $\bar B$ are indeed blocks of $k\bar G$ and $k\bar H$, respectively, because $Z\leq Z(G)\cap Z(H)$. See  \cite[Chapter 5, Theorems 8.10 and 8.11]{NT89}.
Thus, the claim follows  directly from \smallskip Lemma~\ref{thm:rouquier}. 
\item[\rm (b)] We have $\overline{B_0(kG)}=B_0(k\bar G)$, respectively $\overline{B_0(kH)}=B_0(k\bar H)$, and by  Lemma \ref{lem:GxH}(b), $\overline{{\mathrm{Sc}}(G\times H,\Delta P)}= {\mathrm{Sc}}(\bar G\times\bar H,\Delta \bar P)$. Therefore the claim is straightforward from (a).
\end{enumerate}
\end{proof}

\begin{rem}
We note that a weaker version of Theorem~\ref{thm:liftingNew}(a) was used  in \cite[Section 3.5]{UN03} in the context of Brou\'{e}'s abelian defect group conjecture in characteristic $3$ for the group $\SU(3,q^2)$ with $q\equiv 2$ or $5\pmod{9}$.
Moreover,  Theorem~\ref{thm:liftingNew}(b) generalises a previous result \cite[Theorem]{KK05} by the first author and Kunugi in the special case of principal blocks.
\end{rem}


\section{Proof of the main results}

\begin{proof}[Proof of Theorem~\ref{thm:intro}]
Since we are working with principal 2-blocks, we may assume that
$O_{2'}(G)=1$. Thus it follows from the Brauer-Suzuki Theorem \cite[Theorem~2]{Bra64} that the centre $Z:=Z(G)$ of $G$ is cyclic of order 2. Hence $\bar P:=P/Z\cong D_{2^{n-1}}$ 
is a Sylow $2$-subgroup of $\bar G:=G/Z$. \par 
By Proposition~\ref{thm:liftingNew}(b),  the splendid Morita equivalences of Theorem~\ref{dihedral} between principal blocks with dihedral defect groups can be lifted to splendid Morita equivalences between principal blocks with generalised quaternion defect groups.  
Therefore, we need to describe all group extensions of the form
$$1\lra Z\lra G\lra \bar G\lra 1\,,$$
where $\bar G$ is one of the groups occurring in Theorem~\ref{dihedral}(a),$(1)$--$(6)$ and such that a Sylow $2$-subgroup of $G$ is generalised quaternion, that is such that  we have a commutative diagram of the following form:
$$\xymatrix{  1  \ar[r]^-{}   &  Z \ar[r]^-{}  \ar@{=}[d]_-{}   &  G   \ar[r]^-{can}    & \bar G   \ar[r]   &  1  \\
1\ar[r]^-{}   &  Z  \ar[r]^-{}  &  P=Q_{2^n}  \ar[r]^-{can} \ar@{^(->}[u]_-{}   &  \bar P=D_{2^{n-1}} \ar[r]^-{} \ar@{^(->}[u]_{} &   1   }$$ 
By Lemma~\ref{lem:cohomWithCoeffsinC2}, the restriction map 
$$\text{res}^{\bar G}_{\bar P}:H^2(\bar G,C_2)\lra H^2(\bar P,C_2)$$
is injective, and  by Lemma~\ref{lem:centralExtensionsD2n}, there is a unique class of $2$-cocycles in $H^2(D_{2^{n-1}},C_2)\cong(\IZ/2)^3$ such that the middle term of the corresponding extension is isomorphic to $Q_{2^n}$. Therefore, there can be at most one  class of $2$-cocycles in $H^2(\bar G, C_2)$ corresponding to a group extension where the middle term  has a generalised quaternion 
Sylow 2-subgroup. 
Thus, we can now go through the list $(1)$--$(6)$ in Theorem~\ref{dihedral}: If $\bar G=D_{2^{n-1}}$, then $G=Q_{2^n}$. If $n=4$ and $\bar G=\fA_7$, then $G=2.\fA_7$ has a generalised quaternion Sylow $2$-subgroup. 
If $\bar G=\PSL_2(q)$ with $q$ is an odd prime power such that $(q\pm1)_2=2^{n-1}$, then 
$\SL_2(q)$ has a generalised quaternion Sylow $2$-subgroup. If $\bar G=\PGL_2(q)$ with $q$ is an odd prime power such that $(q\pm1)_2=2^{n-1}$, then 
$2.\PGL_2(q)$ has a generalised quaternion Sylow $2$-subgroup.  This proves (a).\par
Part (b) and 
Part (c) are then also straightforward from Theorem~\ref{dihedral}(b) and (c) together with Proposition~\ref{thm:liftingNew}(b). 
\end{proof}
\bigskip

\section*{Acknowledgments}
\noindent
{\small The authors would like to thank Gunter Malle for his notes and comments on an early version of this paper,  and the referee for their careful reading of the manuscript and useful suggestions.}

\bigskip

\end{document}